\documentclass{article}
\usepackage{amsfonts,amsmath,amssymb,amsthm}

\usepackage{hyperref}

\usepackage{graphics, epsfig}

\usepackage{titling}
\predate{}
\postdate{}

\title{Dimension computations for tropical determinantal varieties and prevarieties}
\author{Dylan Zwick \\ \href{mailto:zwick@math.utah.edu}{zwick@math.utah.edu}}
\date{}

\begin{document}

\maketitle

\begin{abstract} 
  This paper proves that when the $r \times r$ minors of an $m \times n$ matrix of indeterminates are not a tropical basis then the tropical prevariety has greater dimension than the tropical variety. It proves the same for the $r \times r$ minors of an $n \times n$ symmetric matrix of indeterminates when $r > 4$.
\end{abstract}

A natural question to ask when a basis is not a tropical basis is how much larger the tropical prevariety is than the tropical variety. In particular, whether the basis ``fails big'', and the tropical prevariety has greater dimension than the tropical variety. In this paper, we prove this is the case whenever the $r \times r$ minors of an $m \times n$ matrix of indeterminates are not a tropical basis, and whenever the $r \times r$ minors of an $n \times n$ symmetric matrix of indeterminates are not a tropical basis and $n > 4$. 

In the first section, we review the foundational ideas from tropical geometry used in this paper. The second section proves the result. All statements about tropical ranks and symmetric tropical ranks for specific matrices, and in fact all specific computational claims of any sort, made in this paper can be verified using Maple code available online: 

http://www.math.utah.edu/\url{~}zwick/Dissertation/.

The author would like to thank the mathematics department of the University of Utah for support during the research for this paper, and in particular his advisor Aaron Bertram. He would also like to thank Melody Chan for pointing out the cocircuit matrix of the Fano matroid could be rearranged to be symmetric.

\section{Tropical preliminaries}

This section introduces the basic ideas from tropical geometry used in this paper, and reviews relevant results from general and symmetric matrices.

\subsection{Tropical basics}

The \emph{tropical semiring} $(\mathbb{R}, \oplus, \odot)$, is defined as the semiring with arithmetic operations:
\begin{center}
  $a \oplus b := min(a,b)$ \hspace{.1 in} and \hspace{.1 in} $a \odot b := a + b$.
\end{center}

A \emph{tropical monomial} $X_{1}^{a_{1}} \cdots X_{m}^{a_{m}}$ is a symbol, and represents a function equivalent to the linear form $\sum_{i} a_{i}X_{i}$ (standard addition and multiplication). 

A \emph{tropical polynomial} is a tropical sum of tropical monomials  
  \begin{center}
    $F(X_{1},\ldots,X_{m}) := \bigoplus_{a \in \mathcal{A}}C_{a}X_{1}^{a_{1}}X_{2}^{a_{2}} \cdots X_{m}^{a_{m}}$, \hspace{.1 in} with $\mathcal{A} \subset \mathbb{N}^{m}$, $C_{a} \in \mathbb{R}$
  \end{center}
  (tropical addition and multiplication), and represents a piecewise linear convex function $F: \mathbb{R}^{m} \rightarrow \mathbb{R}$. 

In this paper, tropical polynomials will be represented with upper case letters, while standard polynomials will be lower case.

The \emph{tropical hypersurface} $\textbf{V}(F)$ defined by a tropical polynomial $F$ is the locus of points $P \in \mathbb{R}^{m}$ such that at least two monomials in $F$ are minimal at $P$. This is also called the \emph{double-min locus} of $F$.

For example, the tropical hypersurface defined by the tropical polynomial
\begin{center}
  $X \oplus Y \oplus 0  = min\{x,y,0\}$
\end{center}
would include the point $(1,0)$, as both $Y$ and $0$ are minimal at that point, but would not include the point $(-1,0)$, as $X$ is uniquely minimal at that point. This is an example of a tropical line.

\vspace{.1 in}
\begin{tabular}{c}
  \centering
  \hspace{1in}\includegraphics[scale=1]{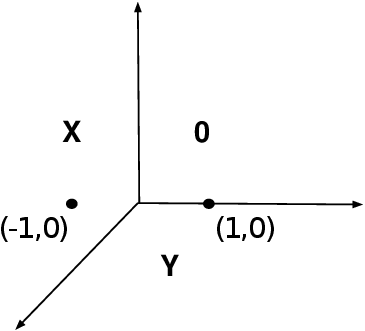}
\end{tabular}

\subsection{Tropical bases}

Let $k$ be an algebraically closed field. Let $f \in k[x_{1},\ldots,x_{m}]$ be a polynomial. The locus of points $p \in k^{m}$ such that $f(p) = 0$ is a \emph{hypersurface}, and is denoted $\textbf{V}(f)$. Let $I$ be an ideal of $k[x_{1},\ldots,x_{m}]$. The ideal $I$ defines a \emph{algebraic variety}, (or \emph{variety}, for short) $\textbf{V}(I)$, in $k^{m}$, which is the set of points $p \in k^{m}$ such that $f(p) = 0$ for all $f \in I$. If $I = (f_{1},\ldots,f_{n})$  then the set $\{f_{1},\ldots,f_{n}\}$ is a \emph{basis} for $I$, and $\textbf{V}(I)$ is equal to the locus of points $p \in k^{m}$ such that $f_{i}(p) = 0$ for all $f_{i}$ in the basis. Put succinctly
\begin{center}
  $\textbf{V}(I) = \bigcap \textbf{V}(f_{i})$.
\end{center}
So, a variety is an intersection of hypersurfaces. By the Hilbert basis theorem every ideal of $k[x_{1},\ldots,x_{m}]$ is finitely generated, so any variety is a finite intersection of hypersurfaces.

In the tropical setting there is an analog of a hypersurface, and we would like an analog of a variety. It might seem natural to define a tropical variety as the intersection of a finite set of tropical hypersurfaces, but these sets do not always have the properties we need in order for them to be useful analogs of algebraic varieties, and we instead call these sets tropical prevarieties. 

A \emph{tropical prevariety} $\textbf{V}(F_{1},\ldots,F_{n})$ is a finite intersection of tropical hypersurfaces:  
\begin{center}
  $\textbf{V}(F_{1},\ldots,F_{n}) = \bigcap_{i = 1}^{n} \textbf{V}(F_{i})$.
\end{center}

A tropical variety is defined differently. Let $K = \mathbb{C}\{\{t\}\}$ be the set of formal power series $a = c_{1}t^{a_{1}} + c_{2}t^{a_{2}} + \cdots$, where $a_{1} < a_{2} < a_{3} < \cdots$ are rational numbers that have a common denominator. These are called Puiseux series, and this set is an algebraically closed field of characteristic zero. For any nonzero element $a \in K$ define the degree of $a$ to be the value of the leading exponent $a_{1}$. This gives us a degree map $deg : K^{*} \rightarrow \mathbb{Q}$. For any two elements $a,b \in K^{*}$ we have
\begin{center}
  $deg(ab) = deg(a) + deg(b) = deg(a) \odot deg(b)$.
\end{center}
Generally, we also have
\begin{center}
  $deg(a + b) = min(deg(a),deg(b)) = deg(a) \oplus deg(b)$.
\end{center}
The only case when this addition relation is not true is when $a$ and $b$ have the same degree, and the coefficients of the leading terms cancel.

We would like to do tropical arithmetic over $\mathbb{R}$, and not just over $\mathbb{Q}$, so we enlarge the field of Puisieux series to allow this. Define the set $\tilde{K}$ by
\begin{center}
  $\tilde{K} = \left\{\sum_{\alpha \in A} c_{\alpha}t^{\alpha} | A \subset \mathbb{R} \text{ well-ordered}, c_{\alpha} \in \mathbb{C}\right\}$.
\end{center}
This is the set of Hahn series, and it is an algebraically closed field of characteristic zero containing the Puisieux series. We define a tropical variety in terms of a variety over $\tilde{K}$.

The degree map on $(\tilde{K}^{*})^{m}$ is the map $\mathcal{T}$ taking points $(p_{1},\ldots,p_{m}) \in (\tilde{K}^{*})^{m}$ to points $(deg(p_{1}),deg(p_{2}),\ldots,deg(p_{m})) \in \mathbb{R}^{m}$. A tropical variety is the image of a variety in $(\tilde{K}^{*})^{m}$ under the degree map. We call this image the \emph{tropicalization} of a set of points in $(\tilde{K}^{*})^{m}$. The tropicalization of a polynomial $f \in \tilde{K}[x_{1},\ldots,x_{m}]$ is the tropical polynomial $\mathcal{T}(f)$ formed by tropicalizing the coefficients of $f$, and converting addition and multiplication into their tropical counterparts. For example, the tropicalization of the polynomial 
\begin{center}
  $f = 3t^{2}xy - 7tx^{3}$ 
\end{center}
is the tropical polynomial
\begin{center}
  $\mathcal{T}(f) = 2XY \oplus 1X^{3}$.
\end{center}

In an unpublished manuscript, Mikhail Kapranov proved the following useful and fundamental result.

\newtheorem{thm}{Theorem}
\begin{thm}[Kapranov's Theorem]
  For $f \in \tilde{K}[x_{1},\ldots,x_{m}]$ the tropical variety $\mathcal{T}(\textbf{V}(f))$ is equal to the tropical hypersurface $\textbf{V}(\mathcal{T}(f))$ determined by the tropical polynomial $\mathcal{T}(f)$.
\end{thm}

Given Kapranov's theorem if $I = (f_{1},\ldots,f_{n})$, then obviously the tropical prevariety determined by the set of tropical polynomials $\{\mathcal{T}(f_{1}),\ldots,\mathcal{T}(f_{n})\}$ contains the tropical variety determined by $I$:
\begin{center}
  $\mathcal{T}(\textbf{V}(I)) \subseteq \bigcap_{i = 1}^{n} \textbf{V}(\mathcal{T}(f_{i}))$.
\end{center}

While Kapranov's theorem gives us the two sets are equal if $n = 1$, in general the containment may be strict. For example, the lines in $(\tilde{K}^{*})^{2}$ defined by the linear equations
\begin{center}
  $f = 2x + y + 1$, \hspace{.1 in} and \hspace{.1 in} $g = tx + ty + 1$,
\end{center}
intersect at the point $(t^{-1}-1,-2t^{-1}+1)$. The tropicalization of this point is $(-1,-1)$, and so if $I = (f,g)$ then
\begin{center}
  $\mathcal{T}(\textbf{V}(I)) = (-1,-1)$.
\end{center}
However, is we tropicalize the linear equations we get:
\begin{center}
  $\mathcal{T}(f) = X \oplus Y \oplus 0$, \hspace{.1 in} and \hspace{.1 in} $\mathcal{T}(g) = 1X \oplus 1Y \oplus 0$.
\end{center}
Each of $\textbf{V}(\mathcal{T}(f))$ and $\textbf{V}(\mathcal{T}(g))$ is a tropical line, and their intersection is the tropical prevariety consisting of all points $(a,a)$ with $a \leq -1$.

\vspace{.1 in}
\begin{tabular}{c}
  \centering
  \hspace{1in}\includegraphics[scale=1]{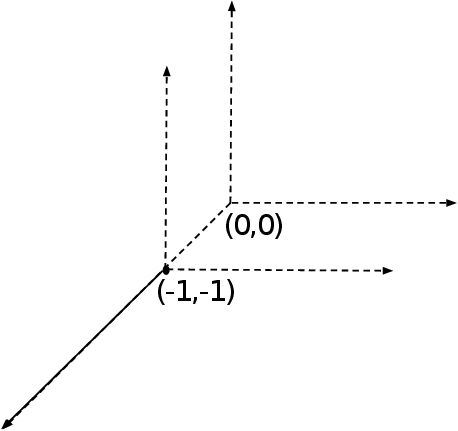}
\end{tabular}

This tropical prevariety properly contains the tropical variety $(-1,-1)$, but the prevariety is clearly much larger. That the intersection of two distinct tropical lines is not necessarily a point is a motivating example of why we do not define a tropical variety to be a finite intersection of tropical hypersurfaces.

\subsection{Kapranov and tropical Rank}

In \cite{dss}, Develin, Santos, and Sturmfels define three notions of matrix rank coming from tropical geometry: the Barvinok rank, the Kapranov rank, and the tropical rank. In this paper we focus on the symmetric analogs of the Kapranov and tropical ranks.

The \emph{tropical rank} of an $m \times n$ matrix $A = (A_{i,j}) \in \mathbb{R}^{m \times n}$ is the smallest number $r \leq min(m,n)$ such that $A$ is not in the tropical prevariety formed by the $r \times r$ minors of an $m \times n$ matrix of indeterminates.

A \emph{lift} of the matrix $A$ is a matrix $\tilde{A} = (\tilde{a}_{i,j}) \in (\tilde{K}^{*})^{m \times n}$ such that $deg(\tilde{a}_{i,j}) = A_{i,j}$ for all $i,j$. The \emph{Kapranov rank} of a matrix is the smallest rank of any lift of the matrix. Equivalently, the Kapranov rank is the smallest number $r \leq min(m,n)$ such that $A$ is not in the tropical variety formed by the $r \times r$ minors of an $m \times n$ matrix of indeterminates.

A square matrix $A = (A_{i,j}) \in \mathbb{R}^{n \times n}$ is \emph{tropically singular} if the minimum
  \begin{center}
    $tropdet(A) := \bigoplus_{\sigma \in S_{n}} A_{1,\sigma(1)} \odot A_{2,\sigma(2)} \odot \cdots \odot A_{n,\sigma(n)}$
  \end{center}
  is attained at least twice is the tropical sum. Here $S_{n}$ denotes the symmetric group on $\{1,2,\ldots,n\}$. We call the number $tropdet(A)$ defined above the \emph{tropical determinant} of $A$, and we say any permutation $\sigma$ such that
  \begin{center}
    $tropdet(A) = A_{1,\sigma(1)} \odot A_{2,\sigma(2)} \odot \cdots \odot A_{n,\sigma(n)}$
  \end{center}
  \emph{realizes} the tropical determinant. So, equivalently, a square matrix $A$ is tropically singular if more than one permutation realizes the tropical determinant.

More generally, suppose $A$ is an $m \times n$ real matrix, and $\{i_{1},i_{2},\ldots,i_{r}\}$ and $\{j_{1},j_{2},\ldots,j_{r}\}$ are subsets of $\{1,\ldots,m\}$ and $\{1,\ldots,n\}$, respectively. These subsets define an $r \times r$ submatrix $A'$ of $A$, with row indices $\{i_{1},\ldots,i_{r}\}$ and column indices $\{j_{1},\ldots,j_{r}\}$. A tropical monomial of the form
\begin{center}
  $\bigodot_{k = 1}^{r}X_{i_{k},\rho(i_{k})}$,
\end{center}
where $\rho$ is a bijection from the row indices to the column indices, is a \emph{minimizing monomial} for the submatrix $A'$ if, over all monomials defined by bijections from $\{i_{1},i_{2},\ldots,i_{r}\}$ to $\{j_{1},j_{2},\ldots,j_{r}\}$, this monomial is minimal under the valuation $X_{i,j} \mapsto A_{i,j}$. An $r \times r$ submatrix of $A$ is tropically singular if it has more than one minimizing monomial.

The tropical variety defined by the $r \times r$ minors of an $m \times n$ matrix of indeterminates is contained in the tropical prevariety defined by the same minors, and therefore
\begin{center}
  tropical rank $\leq$ Kapranov rank.
\end{center}

A natural question to ask about Kapranov and tropical rank is when they are necessarily equal. In other words, for what values $r,m,n$ does tropical rank $r$ imply Kapranov rank $r$ for any $m \times n$ matrix.

This question was answered through the combined work of Develin, Santos, and Sturmfels \cite{dss}, Chan, Jensen, and Rubei \cite{cjr}, and Shitov \cite{sh2}. The result is named after Shitov \cite{ms}, as he completed the project.

\begin{thm}[Shitov's Theorem]
  The $r \times r$ minors of an $m \times n$ matrix of indeterminates form a tropical basis if and only if:
  \begin{itemize}
    \item $r = min(m,n)$, or
    \item $r \leq 3$, or
    \item $r = 4$ and $min(m,n) \leq 6$.
  \end{itemize}
\end{thm}

Denote the set of $m \times n$ matrices with tropical rank $r$ by $T_{m,n,r}$, and the set of $m \times n$ matrices with Kapranov rank $r$ by $\tilde{T}_{m,n,r}$.

\subsection{Symmetric Kapranov and symmetric tropical rank}

The symmetric Kapranov and symmetric tropcial ranks are defined analogously to their general counterparts.

The \emph{symmetric tropical rank} of an $n \times n$ symmetric matrix $A = (A_{i,j}) \in \mathbb{R}^{n \times n}$ is the smallest number $r \leq n$ such that $A$ is not in the tropical prevariety formed by the $r \times r$ minors of an $n \times n$ symmetric matrix of indeterminates.

The \emph{symmetric Kapranov rank} of an $n \times n$ symmetric matrix $A = (A_{i,j}) \in \mathbb{R}^{n \times n}$ is the smallest rank of any lift to a symmetric matrix. Equivalently, the symmetric Kapranov rank is the smallest number $r \leq n$ such that $A$ is not in the tropical variety formed by the $r \times r$ minors of an $n \times n$ symmetric matrix of indeterminates.

The tropical variety defined by the $r \times r$ minors of an $n \times n$ symmetric matrix of indeterminates is contained in the tropical prevariety defined by the same minors, and therefore
\begin{center}
  symmetric tropical rank $\leq$ symmetric Kapranov rank.
\end{center}

As a lift to an $n \times n$ symmetric matrix is a lift to an $n \times n$ matrix, we must have
\begin{center}
  Kapranov rank $\leq$ symmetric Kapranov rank.
\end{center}

For symmetric matrices, we say a square submatrix is \emph{symmetrically tropically singular} if it has more than one minimizing monomial given the equivalence $X_{i,j} = X_{j,i}$.

The tropical rank of a matrix can be equivalently defined as the smallest value of $r$ such that the matrix has a tropically nonsingular $r \times r$ submatrix, and similarly the symmetric tropical rank of a symmetric matrix can be equivalently defined as the smallest value of $r$ such that the symmetric matrix has a symmetrically tropically nonsingular $r \times r$ submatrix. 

If an $r \times r$ submatrix of a symmetric $n \times n$ matrix has two distinct minimizing monomials, given the equivalence $X_{i,j} = X_{j,i}$, then a fortiori it has two distinct minimizing monomials without that equivalence, and so
\begin{center}
  tropical rank $\leq$ symmetric tropical rank.
\end{center}

For example, the tropical determinant of a $3 \times 3$ matrix of indeterminates 
\begin{center}
  $\left(\begin{array}{ccc} X_{1,1} & X_{1,2} & X_{1,3} \\ X_{2,1} & X_{2,2} & X_{2,3} \\ X_{3,1} & X_{3,2} & X_{3,3} \end{array}\right)$
\end{center}
is
\begin{center}
  $X_{1,1}X_{2,2}X_{3,3} \oplus X_{1,2}X_{2,3}X_{3,1} \oplus X_{1,3}X_{2,1}X_{3,2} \oplus X_{1,1}X_{2,3}X_{3,2} \oplus X_{1,2}X_{2,1}X_{3,3} \oplus X_{1,3}X_{2,2}X_{3,1}$.
\end{center}
For the matrix
\begin{center}
    $\left(\begin{array}{ccc} 1 & 0 & 0 \\ 0 & 1 & 0 \\ 0 & 0 & 1 \end{array}\right)$
  \end{center}
there are two minimizing monomials, $X_{1,2}X_{2,3}X_{3,1}$ and $X_{1,3}X_{2,1}X_{3,2}$, and so the matrix is tropically singular. However, under the equivalence $X_{i,j} = X_{j,i}$ the tropical determinant becomes
\begin{center}
  $X_{1,1}X_{2,2}X_{3,3} \oplus X_{1,2}X_{2,3}X_{1,3} \oplus X_{1,1}X_{2,3}^{2} \oplus X_{1,2}^{2}X_{3,3} \oplus X_{1,3}^{2}X_{2,2}$,
\end{center}
and the monomial $X_{1,2}X_{2,3}X_{1,3}$ is the unique minimizing monomial, so the matrix is not symmetrically tropically singular.

In \cite{z1} the author proves the following partial analog of Shitov's theorem for symmetric matrices.

\begin{thm}
    The $r \times r$ minors of an $n \times n$ symmetric matrix of indeterminates are a tropical basis if $r = 2$, $r = 3$, or $r = n$. The minors are not a tropical basis if $4 < r < n$, or if $r = 4$ and $n > 12$.
\end{thm} 

Denote the set of $n \times n$ symmetric matrices with symmetric tropical rank $r$ by $S_{n,r}$, and the set of $n \times n$ symmetric matrices with Kapranov rank $r$ by $\tilde{S}_{n,r}$.

\section{The dimensions of determinantal tropical varieties and tropical prevarieties}

The examination of when the minors of a standard matrix do not form a tropical basis begins with a couple of foundational examples. The same is true in the symmetric case. 

\subsection{Foundational examples}

In \cite{dss}, Develin, Santos, and Sturmfels proved the cocircuit matrix of the Fano matroid,
\begin{center}
  $\left(\begin{array}{ccccccc} 1 & 1 & 0 & 1 & 0 & 0 & 0 \\ 0 & 1 & 1 & 0 & 1 & 0 & 0 \\ 0 & 0 & 1 & 1 & 0 & 1 & 0 \\ 0 & 0 & 0 & 1 & 1 & 0 & 1 \\ 1 & 0 & 0 & 0 & 1 & 1 & 0 \\ 0 & 1 & 0 & 0 & 0 & 1 & 1 \\ 1 & 0 & 1 & 0 & 0 & 0 & 1 \end{array}\right)$,
\end{center}
has tropical rank three but Kapranov rank four. If we permute the rows of this matrix with the permutation given by the disjoint cycle decomposition $(27)(36)(45)$ we get the symmetric matrix
\begin{center}  
  $\left(\begin{array}{ccccccc} 1 & 1 & 0 & 1 & 0 & 0 & 0 \\ 1 & 0 & 1 & 0 & 0 & 0 & 1 \\ 0 & 1 & 0 & 0 & 0 & 1 & 1 \\ 1 & 0 & 0 & 0 & 1 & 1 & 0 \\ 0 & 0 & 0 & 1 & 1 & 0 & 1 \\ 0 & 0 & 1 & 1 & 0 & 1 & 0 \\ 0 & 1 & 1 & 0 & 1 & 0 & 0 \end{array}\right)$.
\end{center}
This symmetric matrix has standard tropical rank three, but symmetric tropical rank four, and is therefore \emph{not} an example of a matrix with symmetric tropical rank three but greater symmetric Kapranov rank.

This matrix can, however, be used to construct a symmetric matrix with symmetric tropical rank three, but greater symmetric Kapranov rank:
\begin{center}
  $\left(\begin{array}{ccccccccccccc} 0 & 0 & 0 & 0 & 0 & 0 & 1 & 1 & 0 & 1 & 0 & 0 & 0 \\ 0 & 0 & 0 & 0 & 0 & 0 & 1 & 0 & 1 & 0 & 0 & 0 & 1 \\ 0 & 0 & 0 & 0 & 0 & 0 & 0 & 1 & 0 & 0 & 0 & 1 & 1 \\ 0 & 0 & 0 & 0 & 0 & 0 & 1 & 0 & 0 & 0 & 1 & 1 & 0 \\ 0 & 0 & 0 & 0 & 0 & 0 & 0 & 0 & 0 & 1 & 1 & 0 & 1 \\ 0 & 0 & 0 & 0 & 0 & 0 & 0 & 0 & 1 & 1 & 0 & 1 & 0 \\ 1 & 1 & 0 & 1 & 0 & 0 & 0 & 1 & 1 & 0 & 1 & 0 & 0 \\ 1 & 0 & 1 & 0 & 0 & 0 & 1 & 0 & 0 & 0 & 0 & 0 & 0 \\ 0 & 1 & 0 &0 & 0 & 1 & 1 & 0 & 0 & 0 & 0 & 0 & 0 \\ 1 & 0 & 0 & 0 & 1 & 1 & 0 & 0 & 0 & 0 & 0 & 0 & 0 \\ 0 & 0 & 0 & 1 & 1 & 0 & 1 & 0 & 0 & 0 & 0 & 0 & 0 \\ 0 & 0 & 1 & 1 & 0 & 1 & 0 & 0 & 0 & 0 & 0 & 0 & 0 \\ 0 & 1 & 1 & 0 & 1 & 0 & 0 & 0 & 0 & 0 & 0 & 0 & 0 \end{array}\right)$
\end{center}
The upper-right, and bottom-left, $7 \times 7$ submatrices of the above $13 \times 13$ symmetric matrix are the symmetric version of the cocircuit matrix of the Fano matroid. This $13 \times 13$ matrix has symmetric tropical rank three. If it had symmetric Kapranov rank three then its upper-right $7 \times 7$ submatrix would have standard Kapranov rank three, and this is impossible.

In \cite{sh1} the matrix
\begin{center}
  $\left(\begin{array}{cccccc} 0 & 0 & 4 & 4 & 4 & 4 \\ 0 & 0 & 2 & 4 & 1 & 4 \\ 4 & 4 & 0 & 0 & 4 & 4 \\ 2 & 4 & 0 & 0 & 2 & 4 \\ 4 & 4 & 4 & 4 & 0 & 0 \\ 2 & 4 & 1 & 4 & 0 & 0 \end{array}\right)$,
\end{center}
was shown to have tropical rank four but Kapranov rank five. If we permute the rows of this matrix with the permutation $(135)(246)$, and the columns with the permutation $(16)(25)(34)$, we get the symmetric matrix
\begin{center}
  $\left(\begin{array}{cccccc} 0 & 0 & 2 & 4 & 1 & 4 \\ 0 & 0 & 4 & 4 & 4 & 4 \\ 2 & 4 & 2 & 4 & 0 & 0 \\ 4 & 4 & 4 & 4 & 0 & 0 \\ 1 & 4 & 0 & 0 & 2 & 4 \\ 4 & 4 & 0 & 0 & 4 & 4 \end{array}\right)$.
\end{center}
This symmetric $6 \times 6$ matrix has symmetric tropical rank four, and, as its Kapranov rank is five, its symmetric Kapranov rank is at least five. Applying Theorem 3 we see its symmetric Kapranov rank is exactly five. So, it is a $6 \times 6$ symmetric matrix with different symmetric tropical and symmetric Kapranov ranks.

\subsection{Dimension growth of determinantal prevarieties}

If a basis for an ideal is not a tropical basis, a natural question to ask is whether the corresponding tropical prevariety has greater dimension than the corresponding tropical variety. In the context of determinantal varieties, this question is whether when the containment
\begin{center}
  $\tilde{T}_{m,n,r} \subseteq T_{m,n,r}$,
\end{center}
is proper, the inequality
\begin{center}
  $dim(\tilde{T}_{m,n,r}) \leq dim(T_{m,n,r})$,
\end{center}
is strict. For symmetric matrices we can ask the analogous question, namely, whether when the containment
\begin{center}
  $\tilde{S}_{n,r} \subseteq S_{n,r}$
\end{center}
is proper the inequality
\begin{center}
  $dim(\tilde{S}_{n,r}) \leq dim(S_{n,r})$
\end{center}
is strict.

In this section we prove the answer for standard matrices is yes, and for symmetric matrices the answer is yes for all cases outside rank three. Note that this answer in the case of standard matrices seems to be known to the mathematical community \cite{c}, but I am unaware of a source for a proof outside this paper.

The proofs for the standard and the symmetric cases are similar, and so will be given in parallel. The proofs are inductive, and will rely upon applying preliminary lemmas to specific base cases. We first prove these lemmas, then examine the base cases, and finally prove the main theorems. We begin, in this subsection, with the lemmas.

Our first lemma concerns tropical linear combinations of tropically linearly independent columns, and could be viewed as a corollary of Theorem 4.2 from \cite{dss}.

\newtheorem{lemma}{Lemma}
\begin{lemma}
  If $A$ is an $r \times r$ tropically nonsingular matrix and the permutation $\sigma \in S_{r}$ realizes the tropical determinant, then there exist constants $c_{1},\ldots,c_{r}$ such that
  \begin{center}
    $c_{\sigma(i)} \odot a_{i,\sigma(i)} \leq c_{j} \odot a_{i,j}$;
  \end{center}
  for all $i,j \leq r$, with equality if and only if $\sigma(i) = j$.
\end{lemma}

\begin{proof}
  Denote the columns of $A$ by $\textbf{a}_{1},\ldots,\textbf{a}_{r}$. As the tropical rank of $A$ is $r$, by Theorem 4.2 of \cite{dss} the dimension of the tropical convex hull of the columns of $A$ is $r$.\footnote{Note that we view the tropical convex hull as a subset of $\mathbb{R}^{r}$, and not of $\mathbb{TP}^{r-1}$, which is the reason the dimension is $r$ here and not $r-1$.} In particular, if we choose $c_{1},\ldots,c_{r}$ such that
  \begin{center}
    $c_{1} \odot \textbf{a}_{1} \oplus c_{2} \odot \textbf{a}_{2} \oplus \cdots \oplus c_{r} \odot \textbf{a}_{r}$
  \end{center}
  is in the interior of the tropical convex hull, then any small modification of a coefficient $c_{i}$ must change the corresponding point in the convex hull. This requires that there exists a permutation $\rho \in S_{r}$ such that $c_{i} + a_{\rho(i),i} \leq c_{k} + a_{\rho(i),k}$ for all $k \leq r$, with equality if and only if $i = k$. The sum of these $a_{\rho(i),i}$ terms must be the determinant, and our lemma is proved with $\sigma = \rho^{-1}$.
\end{proof}

We now present, in both the standard and symmetric cases, how given a matrix $A$ with tropical or symmetric tropical rank $r$, we can construct larger matrices from $A$ with desired tropical or symmetric tropical ranks. We begin with the standard case.

\subsection{The standard case}

\begin{lemma}
  Suppose $A$ is an $m \times n$ matrix with tropical rank $r$. Construct the $m \times (n+1)$ matrix $A'$ from $A$ by appending to $A$ a column formed as a tropical linear combination of columns from $A$. The matrix $A'$ has tropical rank $r$. If we construct the $(m+1) \times n$ matrix $A''$ from $A$ by appending to $A$ a row formed as a tropical linear combination of rows from $A$, then the matrix $A''$ has tropical rank $r$ as well.
\end{lemma}

\begin{proof}
  As column $n+1$ of $A'$ is a tropical linear combination of the columns of $A$, the tropical convex hull of the columns of $A'$ is the same as the tropical convex hull of the columns of $A$, and therefore by Theorem 4.2 from \cite{dss} the two matrices have the same tropical rank.
  An identical argument, mutatis mutandis, proves $A''$ has tropical rank $r$.
\end{proof}

\begin{lemma}
  Suppose $A$ is an $m \times n$ matrix with tropical rank $r$. Construct the $(m+1) \times (n+1)$ matrix $A'$ from $A$ by choosing a number $P$ that is greater than any entry of $A$, a number $M$ that is less than any entry of $A$, and defining
  \begin{center}
    $A' = \left(\begin{array}{ccc|c} & & & P \\ & A & & \vdots \\ & & & P \\ \hline P & \cdots & P & M \end{array}\right)$.
  \end{center}
  The matrix $A'$ has tropical rank $r+1$.
\end{lemma}

\begin{proof}
  As $A$ has tropical rank $r$ there is an $r \times r$ submatrix of $A$ that is tropically nonsingular. Let $a_{1},\ldots,a_{r}$ denote the rows of $A$ that define this submatrix, $b_{1},\ldots,b_{r}$ denote the columns of $A$ that define this submatrix, and $D$ denote the submatrix's tropical determinant. The tropical determinant of the $(r+1) \times (r+1)$ submatrix of $A'$ defined by the rows $a_{1},\ldots,a_{r},a_{m+1}$, and the columns $b_{1},\ldots,b_{r},b_{n+1}$ must, given the definitions of $P$ and $M$, be equal to $D \odot M$, and the submatrix must be nonsingular. So, the tropical rank of $A'$ must be at least $r+1$.
  
  Take any $(r+2) \times (r+2)$ submatrix of $A'$. If it is a submatrix of $A$ then, as $A$ has tropical rank $r$, it must be singular. If the submatrix is formed from row $m+1$ of $A'$, but not column $n+1$, then we can see it must be tropically singular by taking a row expansion along the submatrix's bottom row, and noting that every $(r+1) \times (r+1)$ submatrix of $A$ is tropically singular. Similarly, if the submatrix is formed from column $n+1$ of $A'$, but not row $m+1$, the submatrix must be tropically singular. Finally, if the submatrix is formed from row $m+1$ and column $n+1$ then, given the definitions of $P$ and $M$, every tropical product of terms that equals the tropical determinant must involve the term $a_{m+1,n+1} = M$, and singularity of the $(r+2) \times (r+2)$ submatrix follows from the fact that every $(r+1) \times (r+1)$ submatrix of $A$ is tropically singular. So, the tropical rank of $A'$ is at most $r+1$, and combining this with the result from the previous paragraph we see the tropical rank equals $r+1$.
\end{proof}

\subsection{The symmetric case}

The corresponding lemmas for symmetric matrices are similar. However, for the symmetric version of Lemma 2 we do not have a corresponding convenient reference like Theorem 4.2 from \cite{dss}, and consequently the proof is much longer and more involved.

\begin{lemma}
  Suppose $A$ is an $n \times n$ symmetric matrix with symmetric tropical rank $r$. Construct the $n \times (n+1)$ matrix $A'$ from $A$ by appending to the right of $A$ a column formed as a tropical linear combination of columns from $A$. So, if $\textbf{a}_{1},\ldots,\textbf{a}_{n}$ are the columns of $A$ and $\textbf{a}_{n+1}'$ is column $n+1$ of $A'$, then
  \begin{center}
    $\textbf{a}_{n+1}' = c_{i_{1}} \odot \textbf{a}_{i_{1}} \oplus c_{i_{2}} \odot \textbf{a}_{i_{2}} \oplus \cdots \oplus c_{i_{k}} \odot \textbf{a}_{i_{k}}$.
  \end{center}
  Construct the $(n+1) \times (n+1)$ matrix $A''$ from $A'$ by appending to the bottom of $A'$ a row formed as a tropical linear combination of rows from $A'$ in the same manner. So, if $\textbf{a}_{1}',\ldots,\textbf{a}_{n+1}'$ are the rows of $A'$ and $\textbf{a}_{n+1}''$ is column $n+1$ of $A''$, then
  \begin{center}
    $\textbf{a}_{n+1}'' = c_{i_{1}} \cdot \textbf{a}_{i_{1}}' \oplus c_{i_{2}} \cdot \textbf{a}_{i_{2}}' \oplus \cdots \oplus c_{i_{k}} \odot \textbf{a}_{i_{k}}'$.
  \end{center}
  The matrix $A''$ is symmetric, and has symmetric tropical rank $r$.
\end{lemma}

\begin{proof}
  The entry $a_{j,n+1}''$ of $A''$, where $j < n+1$, is a tropical linear combination of elements from row $j$ of $A$:
  \begin{center}
    $a_{j,n+1}'' = c_{k_{1}} \odot a_{j,k_{1}} \oplus c_{k_{2}} \odot a_{j,k_{2}} \oplus \cdots \oplus c_{k_{l}} \odot a_{j,k_{l}}$.
  \end{center}

  The entry $a_{n+1,j}''$ is similarly a tropical linear combination of elements from column $j$ of $A$:
  \begin{center}
    $a_{n+1,j}'' = c_{k_{1}} \odot a_{k_{1},j} \oplus c_{k_{2}} \odot a_{k_{2},j} \oplus \cdots \oplus c_{k_{l}} \odot a_{k_{l},j}$.
  \end{center}
  As $A$ is symmetric we see immediately that $a_{j,n+1}'' = a_{n+1,j}''$, and therefore $A''$ is also symmetric.
  
  Suppose $M$ is an $(r+1) \times (r+1)$ submatrix of $A''$ that inherits its row and column indices from $A''$. Denote the row indices of $M$ in ascending order by $i_{1},i_{2},\ldots,i_{r+1}$, and the column indices in ascending order by $j_{1},j_{2},\ldots,j_{r+1}$. Denote by $\textbf{m}_{j}$ the column vector formed by rows $i_{1},\ldots,i_{r+1}$ and column $j$ of $A''$. So,
  \begin{center}
    $M = \left(\begin{array}{cccc} & & & \\ \textbf{m}_{j_{1}} & \textbf{m}_{j_{2}} & \cdots & \textbf{m}_{j_{r+1}} \\ & & & \end{array}\right)$.
  \end{center}
  If $M$ does not have a column $n+1$ or row $n+1$ then $M$ corresponds with a submatrix of $A$. In this case as $A$ has symmetric tropical rank $r$, $M$ must be symmetrically tropically singular.
  
  Suppose $M$ has a column $n+1$, but no row $n+1$. There exists a bijection $\sigma$ from the column indices of $M$ to its row indices such that
  \begin{center}
    $tropdet(M) = m_{\sigma(j_{1}),j_{1}} \odot m_{\sigma(j_{2}),j_{2}} \odot \cdots \odot m_{\sigma(j_{r}),j_{r}} \odot m_{\sigma(n+1),n+1}$

    $= a_{\sigma(j_{1}),j_{1}}'' \odot a_{\sigma(j_{2}),j_{2}}'' \odot \cdots \odot a_{\sigma(j_{r}),j_{r}}'' \odot a_{\sigma(n+1),n+1}''$.
  \end{center}
  Note that this bijection $\sigma$ is not necessarily unique.
  
  We know from the construction of $A''$ that $a_{\sigma(n+1),n+1}'' = c_{k_{i}} \odot a_{\sigma(n+1),k_{i}}''$ for some index $k_{i} < n+1$. Using this information, define the matrix
  \begin{center}
    $M' = \left(\begin{array}{ccccc} & & & & \\ \textbf{m}_{j_{1}} & \textbf{m}_{j_{2}} & \cdots & \textbf{m}_{j_{r}} & c_{k_{i}} \odot \textbf{m}_{k_{i}} \\ & & & & \end{array}\right)$.
  \end{center}
  Index the rows and columns of $M'$ with the same indices as $M$. The matrices $M$ and $M'$ differ only in their rightmost column, and $m_{i,n+1} \leq m_{i,n+1}'$ for all entries in their respective rightmost columns. Therefore, $tropdet(M) \leq tropdet(M')$, and it follows immediately that
  \begin{center}
    $tropdet(M) = m_{\sigma(j_{1}),j_{1}} \odot m_{\sigma(j_{2}),j_{2}} \odot \cdots \odot m_{\sigma(j_{r}),j_{r}} \odot m_{\sigma(n+1),n+1}$
    
    $= a_{\sigma(j_{1}),j_{1}}'' \odot a_{\sigma(j_{2}),j_{2}}'' \odot \cdots \odot a_{\sigma(j_{r}),j_{r}}'' \odot a_{\sigma(n+1),n+1}''$
    
    $= m_{\sigma(j_{1}),j_{1}}' \odot m_{\sigma(j_{2}),j_{2}}' \odot \cdots \odot m_{\sigma(j_{r}),j_{r}}' \odot m_{\sigma(n+1),n+1}' = tropdet(M')$.
  \end{center}
  So, $tropdet(M) = tropdet(M')$, and if $\tau$ is a bijection from $\{j_{1},\ldots,j_{r+1}\}$ to $\{i_{1},\ldots,i_{r+1}\}$ such that
  \begin{center}
    $tropdet(M') = m_{\tau(j_{1}),j_{1}}' \odot m_{\tau(j_{2}),j_{2}}' \odot \cdots \odot m_{\tau(j_{r}),j_{r}}' \odot m_{\tau(j_{r+1}),j_{r+1}}'$,
  \end{center}
  then
  \begin{center}
    $tropdet(M) = m_{\tau(j_{1}),j_{1}} \odot m_{\tau(j_{2}),j_{2}} \odot \cdots \odot m_{\tau(j_{r}),j_{r}} \odot m_{\tau(j_{r+1}),j_{r+1}}$.
  \end{center}
  
  Suppose $k_{i} \in \{j_{1},\ldots,j_{r}\}$. In this case one of the columns of $M'$ is a tropical multiple of another, and by Proposition 2.9 there are two distinct bijections $\sigma_{1}$ and $\sigma_{2}$ such that 
  \begin{center}
    $tropdet(M') = m_{\sigma_{1}(j_{1}),j_{1}}' \odot \cdots \odot m_{\sigma_{1}(j_{r}),j_{r}}' = m_{\sigma_{2}(j_{1}),j_{1}}' \odot \cdots \odot m_{\sigma_{2}(j_{r+1}),j_{r+1}}'$
  \end{center}
  and the monomials
  \begin{center}
    $X_{1} = X_{\sigma_{1}(j_{1}),j_{1}} \odot \cdots \odot X_{\sigma_{1}(j_{r+1}),j_{r+1}}$, \hspace{.1 in} and \hspace{.1 in} $X_{2} = X_{\sigma_{2}(j_{1}),j_{1}} \odot \cdots \odot X_{\sigma_{2}(j_{r+1}),j_{r+1}}$
  \end{center}
  are distinct even given the relation $X_{i,j} = X_{j,i}$. The monomials $X_{1}$ and $X_{2}$ must both be minimizing monomials for the submatrix $M$, and therefore this submatrix is symmetrically tropically singular.
  
  If $k_{i} \notin \{j_{1},\ldots,j_{r}\}$ then suppose $j_{q} < k_{i} < j_{q+1}$. Take the submatrix of $A''$ given by
  \begin{center}
    $M'' = \left(\begin{array}{cccccccc} & & & & & & & \\ \textbf{m}_{j_{1}} & \textbf{m}_{j_{2}} & \cdots & \textbf{m}_{j_{q}} & \textbf{m}_{k_{i}} & \textbf{m}_{j_{q+1}} & \cdots & \textbf{m}_{j_{r}} \\ & & & & & & & \end{array}\right)$
  \end{center}
  where $M''$ inherits its row and column indices from $A''$. Any bijection 
  \begin{center}
    $\sigma'':\{j_{1},\ldots,j_{q},k_{i},j_{q+1},\ldots,j_{r}\} \rightarrow \{i_{1},i_{2},\ldots,i_{r+1}\}$
  \end{center}
  such that
  \begin{center}
    $tropdet(M'') = m_{\sigma''(j_{1}),j_{1}}'' \odot \cdots \odot m_{\sigma''(j_{q}),j_{q}}'' \odot m_{\sigma''(k_{i}),k_{i}}'' \odot m_{\sigma''(j_{q+1}),j_{q+1}}'' \odot \cdots \odot m_{\sigma''(j_{r}),j_{r}}''$
  \end{center}
  corresponds with a bijection $\sigma'$ from $\{j_{1},\ldots,j_{r},j_{r+1}\}$ to $\{i_{1},\ldots,i_{r},i_{r+1}\}$ where $\sigma'(j_{p}) = \sigma''(j_{p})$ for $p < r+1$, $\sigma'(j_{r+1}) = \sigma''(k_{i})$, and
  \begin{center}
    $tropdet(M') = m_{\sigma'(j_{1}),j_{1}}' \odot \cdots \odot m_{\sigma'(j_{r+1}),j_{r+1}}'$.
  \end{center}
  The submatrix $M''$ corresponds with an $(r+1) \times (r+1)$ submatrix of $A$, and therefore is symmetrically tropically singular. This implies there are two distinct bijections from $\{j_{1},\ldots,j_{r+1}\}$ to $\{i_{1},\ldots,i_{r+1}\}$, both of which define the tropical determinant of $M'$ in the way $\sigma'$ did above, and which define two monomials that are distinct even under the equivalence $X_{i,j} = X_{j,i}$. These monomials must be minimizing monomials for the submatrix $M$, and therefore $M$ is symmetrically tropically singular. 

  Identical reasoning applies if $M$ has a row $n+1$, but not a column $n+1$.
  
  If $M$ has both a row $n+1$ and a column $n+1$ then we may define $M'$ exactly as we did above, and if $k_{i} \in \{j_{1},\ldots,j_{r}\}$ then the proof goes through without modification. So, suppose $k_{i} \notin \{j_{1},\ldots,j_{r}\}$. In this case the proof above still goes through without modification, if we just note that $M''$ corresponds with an $(r+1) \times (r+1)$ submatrix of $A''$ with a row $n+1$, but not a column $n+1$, and is therefore symmetrically tropically singular.

  So, every $(r+1) \times (r+1)$ submatrix of $A''$ is symmetrically tropically singular, and therefore $A''$ has symmetric tropical rank at most $r$. As $A$ has symmetric tropical rank $r$ there is an $r \times r$ submatrix of $A$ that is symmetrically tropically nonsingular, and there will be a corresponding submatrix in $A''$. So, $A''$ has symmetric tropical rank $r$.
\end{proof}

\begin{lemma}
  Suppose $A$ is an $n \times n$ symmetric matrix with symmetric tropical rank $r$. Construct the $(n+1) \times (n+1)$ matrix $A'''$ from $A$ by choosing a number $P$ that is greater than any entry of $A$, a number $M$ that is less than any entry of $A$, and defining
  \begin{center}
    $A''' = \left(\begin{array}{ccc|c} & & & P \\ & A & & \vdots \\ & & & P \\ \hline P & \cdots & P & M \end{array}\right)$.
  \end{center}
  The matrix $A'''$ is symmetric and has symmetric tropical rank $r+1$.
\end{lemma}

\begin{proof}
  As $A$ is symmetric $A'''$ is obviously symmetric.
  
  The proof that $A'''$ has tropical rank $r+1$ goes exactly the same as the proof of Lemma 5.3, replacing all the pertinent definitions by their symmetric counterparts.
\end{proof}

\subsection{Dimension growth for standard matrices}

We now prove the lemmas at the heart of this chapter. All concern how the dimensions of the determinantal tropical prevarities grow when the size of the matrix is increased. We begin with general matrices, and then turn to symmetric matrices.

\begin{lemma}
  $(dim(T_{m,n+1,r}) - dim(T_{m,n,r})) \geq r-1$, and $(dim(T_{m+1,n,r})-dim(T_{m,n,r})) \geq r-1$. 
\end{lemma}

\begin{proof}
  Suppose $A$ is an $m \times n$ matrix of tropical rank $r-1$. Permuting the rows and columns of a matrix does not change the tropical rank, and so we may assume that the upper-left $(r-1) \times (r-1)$ submatrix of $A$ is tropically nonsingular, and its determinant is realized by the tropical product of the diagonal terms (the classical trace).
  
  Using $A$, define an $m\times (n+1)$ matrix $A'$ by appending to $A$ a tropical linear combination of the first $r-1$ columns of $A$. By Lemma 1 we can pick the coefficients $c_{1},\ldots,c_{r-1}$ for this linear combination such that $c_{i} \odot a_{ii} < c_{j} \odot a_{ij}$ for all $i,j \leq r-1$ with $i \neq j$. By Lemma 2 this matrix $A'$ will have tropical rank $r-1$. 
  
  Viewing $A$ as a point in $\mathbb{R}^{m \times n}$ we define $T_{A,\epsilon}$ to be the intersection of $T_{m,n,r}$ with $B_{A,\epsilon}$, an $\epsilon$-ball centered at $A$:
  \begin{center}
    $T_{A,\epsilon} = T_{m,n,r} \cap B_{A,\epsilon}$.
  \end{center}
  For $\epsilon$ sufficiently small every matrix in $T_{A,\epsilon}$ will, like $A$, have a nonsingular $(r-1) \times (r-1)$ upper-left submatrix with determinant given by the tropical product of the diagonal terms. Similarly, for sufficiently small $\epsilon$, we can use the coefficients $c_{1},\ldots,c_{r-1}$ to define a matrix $B' \in T_{m,n+1,r}$ for any matrix $B \in T_{A,\epsilon}$, such that $c_{i} \odot b_{ii} < c_{j} \odot b_{ij}$ for all $i,j \leq r-1$ with $i \neq j$. This defines an embedding of $T_{A,\epsilon}$ into $T_{m,n+1,r}$. Call this embedding $T_{A,\epsilon}'$.
  
  Tropically multiplying a column of a matrix by a real number does not change the tropical rank. So, for any matrix $B' \in T_{m,n+1,r}$ we can multiply the first $r-1$ columns by constants $c_{1},\ldots,c_{r-1}$ and obtain another point in $T_{m,n+1,r}$. In this way we construct an $(r-1)$-dimensional linear subspace of $T_{m,n+1,r}$. Call this linear subspace $L'_{B'}$. Suppose $B' \in T_{A,\epsilon}'$, and so $B'$ is the image of a matrix $B \in T_{A,\epsilon}$ under our embedding. The intersection $L_{B'}' \cap T_{A,\epsilon}'$ is just the point $B'$. To see this, suppose there were another point, $C' \in L_{B'}' \cap T_{A,\epsilon}'$. This matrix $C'$ would have to be the image of a matrix $C \in T_{A,\epsilon}$ under our embedding, and $C$ would be given by tropically multiplying the first $(r-1)$ columns of $B$ by the appropriate real numbers. The final column of $C'$ would have to be the same as the final column of $B'$, but this would imply the first $(r-1)$ diagonal entries of $C$ are the same as the first $(r-1)$ diagonal entries of $B$, which would imply all the real number tropical multiples are $0$, which would mean $B = C$, and so $B' = C'$.
  
  From this we get $(dim(T_{m,n+1,r}) - dim(T_{m,n,r})) \geq r-1$, and using identical reasoning we can get $(dim(T_{m+1,n,r}) - dim(T_{m,n,r})) \geq r-1$.
\end{proof}

\begin{lemma}
  $(dim(T_{m+1,n+1,r+1}) - dim(T_{m,n,r})) \geq m+n+1$.
\end{lemma}

\begin{proof}
  For $A \in T_{m,n,r}$ we define $T_{A,\epsilon}$ in exactly the same manner as the previous lemma. By Lemma 5.3, for any matrix $B \in T_{A,\epsilon}$ there is a matrix $B'' \in T_{m+1,n+1,r+1}$ defined by
  \begin{center}
    $B'' = \left(\begin{array}{ccc|c} & & & P \\ & B & & P \\ & & & P \\ \hline P & P & P & M \end{array}\right)$
  \end{center}
  where $P$ is larger than any entry in $A$, and $M$ is smaller than any entry in $A$. For $\epsilon$ sufficiently small, this defines an embedding of $T_{A,\epsilon}$ into $T_{m+1,n+1,r+1}$, where the values of $P$ and $M$ are the same for every matrix in the image of the embedding. Call this embedding $T_{A,\epsilon}''$. 

  As noted in the previous lemma, tropically multiplying a column or row of a matrix by a real number does not change its tropical rank. So, for any matrix $B'' \in T_{A,\epsilon}''$ there is a $m + n + 1$ dimensional subspace of $T_{m+1,n+1,r+1}$ formed by tropically multiplying the rows and columns of $B''$ by real numbers (It is not an $m+n+2$ dimensional subspace because adding the same number to all the columns, and then subtracting that number from all the rows, leaves the matrix unchanged). Call this subspace $L''_{B''}$. The intersection $L''_{B''} \cap T_{A,\epsilon}''$ is just the matrix $B''$. We can see this by noting that for every element of $T_{A,\epsilon}''$ the right column and bottom row are the same, and the only element of $L''_{B''}$ with this given right column and bottom row is the matrix $B''$.
\end{proof}

\subsection{Dimension growth for symmetric matrices}

Lemmas 6 and 7 both focus on a neighborhood of a matrix $A \in T_{m,n,r}$. For the symmetric version of Lemma 6 we will require that our matrix $A \in S_{n,r}$ not only have a symmetrically tropically nonsingular $r \times r$ submatrix, but a \emph{tropically nonsingular} $r \times r$ submatrix.

\begin{lemma}
  Suppose $A \in S_{n,r}$, and $A$ has an $(r-1) \times (r-1)$ submatrix that is tropically nonsingular (not just symmetrically tropically nonsingular). Viewing $A$ as a point in $\mathbb{R}^{\binom{n}{2}}$ define $S_{A,\epsilon}$ to be the intersection of $S_{n,r}$ with $B_{A,\epsilon}$, an $\epsilon$-ball centered at $A$:
  \begin{center}
    $S_{A,\epsilon} = S_{n,r} \cap B_{A,\epsilon}$.
  \end{center}
  For $\epsilon$ sufficiently small we have the relation $(dim(S_{n+1,r})-dim(S_{A,\epsilon})) \geq r-1$.
\end{lemma}

\begin{proof}
  The matrix $A$ has an $(r-1) \times (r-1)$ submatrix that is tropically nonsingular. This submatrix is formed by the row indices $\{i_{1},\ldots,i_{r-1}\}$ and the column indices $\{j_{1},\ldots,j_{r-1}\}$. By Lemma 1 there exists a bijection $\sigma$ from the row indices to the column indices of this submatrix, and coefficients $c_{\sigma(i_{1})},\ldots,c_{\sigma(i_{r-1})}$ such that, for all $k,l \leq r-1$, 
  \begin{center}
    $c_{\sigma(i_{k})} \odot a_{i_{k},\sigma(i_{k})} \leq c_{j_{l}} \odot a_{i_{k},j_{l}}$,
  \end{center}
  with equality if and only if $\sigma(i_{k}) = j_{l}$.
  
  Construct the matrix $A'$ by appending to the right of $A$ the column defined by
  \begin{center}
    $\textbf{a}_{n+1}' = c_{\sigma(i_{1})} \odot \textbf{a}_{\sigma(i_{1})} \oplus \cdots \oplus c_{\sigma(i_{r-1})} \odot \textbf{a}_{\sigma(i_{r-1})}$,
  \end{center}
  and construct the matrix $A''$ by appending to the bottom of $A'$ the row defined as a linear combination of rows from $A'$ in the same manner. By Lemma 4, the matrix $A''$ is symmetric and has symmetric tropical rank $r-1$.

  For $\epsilon$ sufficiently small every matrix in $S_{A,\epsilon}$ will, like $A$, have a tropically nonsingular $(r-1) \times (r-1)$ submatrix with row indices $\{i_{1},\ldots,i_{r-1}\}$ and column indices $\{j_{1},\ldots,j_{r-1}\}$. Furthermore, again for $\epsilon$ sufficiently small, we can use the coefficients $c_{\sigma(i_{1})},\ldots,c_{\sigma(i_{r-1})}$ to define a matrix $B'' \in S_{n+1,r}$ for any matrix $B \in S_{A,\epsilon}$. This defines an embedding of $S_{A,\epsilon}$ into $S_{n+1,r}$. Call this embedding $S_{A,\epsilon}''$.
  
  If we tropically multiply both row $i$ and column $i$ of a symmetric matrix by a real number $c$, then the matrix formed is still symmetric, and has the same symmetric tropical rank as the original matrix. So, for any matrix $B'' \in S_{n+1,r}$ we can tropically multiply rows $j_{1},j_{2},\ldots,j_{r-1}$ by constants $d_{j_{1}},d_{j_{2}},\ldots,d_{j_{r-1}}$, and columns $j_{1},j_{2},\ldots,j_{r-1}$ by the same constants to obtain another point in $S_{n+1,r}$. In this way we construct an $(r-1)$-dimensional linear subspace of $S_{n+1,r}$. Call this linear subspace $L_{B''}''$. 
 
  Suppose $B'' \in S_{A,\epsilon}''$, and so $B''$ is the image of a matrix $B \in S_{A,\epsilon}$ under our embedding. The intersection $L_{B''}'' \cap S_{A,\epsilon}''$ is just the point $B''$. To see this, suppose there were another point $C'' \in L_{B''}'' \cap S_{A,\epsilon}''$. This matrix $C''$ would be the image of a matrix $C \in S_{A,\epsilon}$, and $C$ would be given by tropically multiplying the rows $j_{1},j_{2},\ldots,j_{r-1}$ and the columns $j_{1},j_{2},\ldots,j_{r-1}$ of $B$ by the constants $d_{j_{1}},\ldots,d_{j_{r-1}}$. The $(i_{k},n+1)$ term of the image of $C$ will be 
  \begin{center}
    $(c_{\sigma(i_{k})} \odot a_{i_{k},\sigma(i_{k})}) \odot d_{\sigma(i_{k})} \odot d_{i_{k}}$
  \end{center}
  if $i_{k} \in \{j_{1},\ldots,j_{r-1}\}$, and 
  \begin{center}
    $(c_{\sigma(i_{k})} \odot a_{i_{k},\sigma(i_{k})}) \odot d_{\sigma(i_{k})}$ 
  \end{center}
  if not. The $(i_{k},n+1)$ term of $C'$ will be 
  \begin{center}
    $(c_{\sigma(i_{k})} \odot a_{i_{k},\sigma(i_{k})}) \odot d_{i_{k}}$ 
  \end{center}
  if $i_{k} \in \{j_{1},\ldots,j_{r-1}\}$, and 
  \begin{center}
    $(c_{\sigma(i_{k})} \odot a_{i_{k},\sigma(i_{k})})$ 
  \end{center}
  if not. In either case, for these terms to be equal we must have $d_{\sigma(i_{k})} = 0$, and as this must be true for all row indices $i_{k}$, and as $\sigma$ is a bijection from the row indices to the column indices, we have $d_{j_{1}} = d_{j_{2}} = \cdots = d_{j_{r-1}} = 0$. So, $C = B$, and therefore $C' = B'$.
  
  From this we get $(dim(S_{n+1,r}) - dim(S_{A,\epsilon})) \geq r-1$, and our lemma is proven.
\end{proof}

The symmetric version of Lemma 7 is very similar to its general counterpart.

\begin{lemma}
  $(dim(S_{n+1,r+1}) - dim(S_{n,r})) \geq n+1$.
\end{lemma}

\begin{proof}
  For $A \in S_{n,r}$ we define $S_{A,\epsilon}$ in exactly the same manner as in Lemma 5.9. By Lemma 5.6, for any matrix $A \in S_{n,r}$ there is a matrix $A''' \in S_{n+1,r+1}$ defined by
  \begin{center}
    $A''' = \left(\begin{array}{ccc|c} & & & P \\ & A & & P \\ & & & P \\ \hline P & P & P & M \end{array}\right)$
  \end{center}
  where $P$ is larger than any entry in $A$, and $M$ is smaller than any entry in $A$. For $\epsilon$ sufficiently small, this defines an embedding of $S_{A,\epsilon}$ into $S_{n+1,r+1}$, where the same values of $P$ and $M$ are used for each matrix in the image of this embedding. Call this embedding $S_{A,\epsilon}'''$.

  As noted in Lemma 5.9, tropically multiplying a column and row with the same index by a real number does not change the symmetric tropical rank of a matrix. So, for any matrix $B''' \in S_{A,\epsilon}'''$ there is a $n+1$ dimensional subspace of $S_{n+1,r+1}$ formed by tropically multiplying the rows and columns of $B'''$ with the same indices by real numbers. Call this subspace $L'''_{B'''}$. The intersection $L'''_{B'''} \cap S_{A,\epsilon}'''$ is just the matrix $B'''$. We can see this by noting that for every element of $S_{A,\epsilon}'''$ the right column and bottom row are the same, and the only element of $L'''_{B'''}$ with this given right column and bottom row is the matrix $B'''$.
\end{proof}

\subsection{The base cases}

We now have all the lemmas required to prove the inductive parts of our theorems. We simply require the base cases. In this subsection we will use the foundational examples from subsection 2.1 of this chapter to construct the base cases for our dimension inequalities. To do so, we note that if $A$ is an $n \times n$ singular matrix, with permutations that realize the tropical determinant $\sigma_{1},\sigma_{2},\ldots,\sigma_{k}$, then $A$, viewed as a point in $\mathbb{R}^{n \times n}$, will be on the linear space determined by the linear equations

\begin{center}
  
  $x_{1,\sigma_{1}(1)} + x_{2,\sigma_{1}(2)} + \cdots + x_{n,\sigma_{1}(n)} = x_{1,\sigma_{2}(1)} + x_{2,\sigma_{2}(2)} + \cdots + x_{n,\sigma_{1}(n)}$,
  
  $x_{1,\sigma_{1}(1)} + x_{2,\sigma_{1}(2)} + \cdots + x_{n,\sigma_{1}(n)} = x_{1,\sigma_{3}(1)} + x_{2,\sigma_{3}(2)} + \cdots + x_{n,\sigma_{3}(n)}$,
  
  $\vdots$
  
  $x_{1,\sigma_{1}(1)} + x_{2,\sigma_{1}(2)} + \cdots + x_{n,\sigma_{1}(n)} = x_{1,\sigma_{k}(1)} + x_{2,\sigma_{k}(2)} + \cdots + x_{n,\sigma_{k}(n)}$.
  
\end{center}

If we intersect this linear space with a sufficiently small $\epsilon$-ball in $\mathbb{R}^{n \times n}$ centered at $A$, every point in this intersection will correspond with a matrix having the same minimizing permutations as $A$. The dimension of this intersection will be the dimension of the linear space.

For example, the singular matrix
\begin{center}
  $Q = \left(\begin{array}{ccc} 0 & 0 & 1 \\ 0 & 0 & 1 \\ 1 & 1 & 0 \end{array}\right)$
\end{center}
\vspace{.1 in}
will be on the linear space defined by the linear equation
\begin{center}
  $x_{1,1} + x_{2,2}+ x_{3,3} = x_{1,2} + x_{2,1}+ x_{3,3}$.
\end{center}
Any matrix on this linear space within a sufficiently small $\epsilon$-ball around $Q$ will also be singular, and will have the same minimizing permutations as $Q$. Similarly, the singular matrix
\begin{center}
  $R = \left(\begin{array}{ccc} 0 & 0 & 0 \\ 0 & 1 & 0 \\ 0 & 0 & 1 \end{array}\right)$
\end{center}
will be on the linear space defined by the linear equations
\begin{center}
  $x_{1,1} + x_{2,3} + x_{3,2} = x_{1,2} + x_{2,3} + x_{3,1}$,
  
  $x_{1,1} + x_{2,3} + x_{3,2} = x_{1,3} + x_{2,1} + x_{3,2}$.
\end{center}
Any matrix on this linear space within a sufficiently small $\epsilon$-ball around $R$ will also be singular, and will have the same three minimizing permutations as $R$.

Extending this idea, if $B$ is an $m \times n$ matrix with tropical rank $r-1$, then for every $r \times r$ submatrix the permutations realizing the tropical determinant determine a linear space, and the intersection of the linear spaces determined by all the $r \times r$ submatrices is again a linear space. If we intersect the linear space determined by all $r \times r$ submatrices with a sufficiently small $\epsilon$-ball in $\mathbb{R}^{m \times n}$ centered at $B$, then every point in this intersection will correspond with an $m \times n$ matrix with tropical rank $r-1$, for which every $r \times r$ submatrix has the same minimizing permutations as the corresponding submatrix in $B$. In particular, the dimension of this intersection will be the dimension of the linear space determined by all $r \times r$ submatrices, and the dimension of this linear space cannot be greater than the dimension of the tropical prevariety $T_{m,n,r}$.

Along these lines we examine the matrix
\begin{center}
  $\left(\begin{array}{cccccc} 0 & 0 & 2 & 4 & 2 & 4 \\ 0 & 0 & 4 & 4 & 4 & 4 \\ 2 & 4 & 2 & 4 & 0 & 0 \\ 4 & 4 & 4 & 4 & 0 & 0 \\ 2 & 4 & 0 & 0 & 2 & 4 \\ 4 & 4 & 0 & 0 & 4 & 4 \end{array}\right)$,
\end{center}
the symmetric version of the matrix from \cite{sh1}. The minimizing permutations for each $5 \times 5$ submatrix determine the linear equations:
\begin{center}
  $x_{2,2} + x_{3,5} + x_{4,6} + x_{5,3} + x_{6,4} = x_{2,2} + x_{3,5} + x_{4,6} + x_{5,4} + x_{6,3}$;

  $x_{2,2} + x_{3,5} + x_{4,6} + x_{5,3} + x_{6,4} = x_{2,2} + x_{3,6} + x_{4,5} + x_{5,3} + x_{6,4}$;

  $x_{2,2} + x_{3,5} + x_{4,6} + x_{5,3} + x_{6,4} = x_{2,2} + x_{3,6} + x_{4,5} + x_{5,4} + x_{6,3}$; 

  $x_{2,1} + x_{3,5} + x_{4,6} + x_{5,3} + x_{6,4} = x_{2,1} + x_{3,5} + x_{4,6} + x_{5,4} + x_{6,3}$;

  $x_{2,1} + x_{3,5} + x_{4,6} + x_{5,3} + x_{6,4} = x_{2,1} + x_{3,6} + x_{4,5} + x_{5,3} + x_{6,4}$; 

  $x_{2,1} + x_{3,5} + x_{4,6} + x_{5,3} + x_{6,4} = x_{2,1} + x_{3,6} + x_{4,5} + x_{5,4} + x_{6,3}$;

  $x_{2,2} + x_{3,5} + x_{4,6} + x_{5,1} + x_{6,4} = x_{2,2} + x_{3,6} + x_{4,5} + x_{5,1} + x_{6,3}$; 

  $x_{2,2} + x_{3,5} + x_{4,6} + x_{5,1} + x_{6,3} = x_{2,2} + x_{3,6} + x_{4,5} + x_{5,1} + x_{6,3}$;

  $x_{2,2} + x_{3,1} + x_{4,6} + x_{5,3} + x_{6,4} = x_{2,2} + x_{3,1} + x_{4,6} + x_{5,4} + x_{6,3}$; 

  $x_{2,2} + x_{3,1} + x_{4,5} + x_{5,3} + x_{6,4} = x_{2,2} + x_{3,1} + x_{4,5} + x_{5,4} + x_{6,3}$;

  $x_{1,2} + x_{3,5} + x_{4,6} + x_{5,3} + x_{6,4} = x_{1,2} + x_{3,5} + x_{4,6} + x_{5,4} + x_{6,3}$; 

  $x_{1,2} + x_{3,5} + x_{4,6} + x_{5,3} + x_{6,4} = x_{1,2} + x_{3,6} + x_{4,5} + x_{5,3} + x_{6,4}$;

  $x_{1,2} + x_{3,5} + x_{4,6} + x_{5,3} + x_{6,4} = x_{1,2} + x_{3,6} + x_{4,5} + x_{5,4} + x_{6,3}$; 

  $x_{1,1} + x_{3,5} + x_{4,6} + x_{5,3} + x_{6,4} = x_{1,1} + x_{3,5} + x_{4,6} + x_{5,4} + x_{6,3}$;

  $x_{1,1} + x_{3,5} + x_{4,6} + x_{5,3} + x_{6,4} = x_{1,1} + x_{3,6} + x_{4,5} + x_{5,3} + x_{6,4}$; 

  $x_{1,1} + x_{3,5} + x_{4,6} + x_{5,3} + x_{6,4} = x_{1,1} + x_{3,6} + x_{4,5} + x_{5,4} + x_{6,3}$;

  $x_{1,2} + x_{3,5} + x_{4,6} + x_{5,3} + x_{6,4} = x_{1,2} + x_{3,6} + x_{4,5} + x_{5,1} + x_{6,4}$; 

  $x_{1,2} + x_{3,5} + x_{4,6} + x_{5,1} + x_{6,3} = x_{1,2} + x_{3,6} + x_{4,5} + x_{5,1} + x_{6,3}$;

  $x_{1,2} + x_{3,1} + x_{4,6} + x_{5,3} + x_{6,4} = x_{1,2} + x_{3,1} + x_{4,6} + x_{5,4} + x_{6,3}$; 

  $x_{1,2} + x_{3,1} + x_{4,5} + x_{5,3} + x_{6,4} = x_{1,2} + x_{3,1} + x_{4,5} + x_{5,4} + x_{6,3}$;

  $x_{1,3} + x_{2,2} + x_{4,6} + x_{5,3} + x_{6,4} = x_{1,3} + x_{2,2} + x_{4,6} + x_{5,4} + x_{6,3}$; 

  $x_{1,3} + x_{2,2} + x_{4,6} + x_{5,3} + x_{6,4} = x_{1,3} + x_{2,1} + x_{4,6} + x_{5,4} + x_{6,3}$;

  $x_{1,1} + x_{2,2} + x_{4,6} + x_{5,5} + x_{6,4} = x_{1,2} + x_{2,1} + x_{4,6} + x_{5,5} + x_{6,4}$; 

  $x_{1,1} + x_{2,2} + x_{4,6} + x_{5,5} + x_{6,3} = x_{1,2} + x_{2,1} + x_{4,6} + x_{5,5} + x_{6,3}$;

  $x_{1,1} + x_{2,2} + x_{4,6} + x_{5,3} + x_{6,4} = x_{1,1} + x_{2,2} + x_{4,6} + x_{5,4} + x_{6,3}$; 

  $x_{1,1} + x_{2,2} + x_{4,6} + x_{5,3} + x_{6,4} = x_{1,2} + x_{2,1} + x_{4,6} + x_{5,3} + x_{6,4}$;

  $x_{1,1} + x_{2,2} + x_{4,6} + x_{5,3} + x_{6,4} = x_{1,2} + x_{2,1} + x_{4,6} + x_{5,4} + x_{6,3}$; 

  $x_{1,1} + x_{2,2} + x_{4,5} + x_{5,3} + x_{6,4} = x_{1,1} + x_{2,2} + x_{4,5} + x_{5,4} + x_{6,3}$;

  $x_{1,1} + x_{2,2} + x_{4,5} + x_{5,3} + x_{6,4} = x_{1,2} + x_{2,1} + x_{4,5} + x_{5,3} + x_{6,4}$; 

  $x_{1,1} + x_{2,2} + x_{4,5} + x_{5,3} + x_{6,4} = x_{1,2} + x_{2,1} + x_{4,5} + x_{5,4} + x_{6,3}$;

  $x_{1,3} + x_{2,2} + x_{3,6} + x_{5,3} + x_{6,4} = x_{1,5} + x_{2,2} + x_{3,6} + x_{5,4} + x_{6,3}$; 

  $x_{1,5} + x_{2,1} + x_{3,6} + x_{5,3} + x_{6,4} = x_{1,5} + x_{2,1} + x_{3,6} + x_{5,4} + x_{6,3}$;

  $x_{1,1} + x_{2,2} + x_{3,6} + x_{5,5} + x_{6,4} = x_{1,2} + x_{2,1} + x_{3,6} + x_{5,5} + x_{6,4}$; 

  $x_{1,1} + x_{2,2} + x_{3,6} + x_{5,5} + x_{6,3} = x_{1,2} + x_{2,1} + x_{3,6} + x_{5,5} + x_{6,3}$;

  $x_{1,1} + x_{2,2} + x_{3,6} + x_{5,3} + x_{6,4} = x_{1,2} + x_{2,1} + x_{3,6} + x_{5,3} + x_{6,4}$; 

  $x_{1,1} + x_{2,2} + x_{3,6} + x_{5,3} + x_{6,4} = x_{1,1} + x_{2,2} + x_{3,6} + x_{5,4} + x_{6,3}$;

  $x_{1,1} + x_{2,2} + x_{3,6} + x_{5,3} + x_{6,4} = x_{1,2} + x_{2,1} + x_{3,6} + x_{5,4} + x_{6,3}$; 

  $x_{1,1} + x_{2,2} + x_{3,5} + x_{5,3} + x_{6,4} = x_{1,2} + x_{2,1} + x_{3,5} + x_{5,3} + x_{6,4}$;

  $x_{1,1} + x_{2,2} + x_{3,5} + x_{5,3} + x_{6,4} = x_{1,1} + x_{2,2} + x_{3,5} + x_{5,4} + x_{6,3}$; 

  $x_{1,1} + x_{2,2} + x_{3,5} + x_{5,3} + x_{6,4} = x_{1,2} + x_{2,1} + x_{3,5} + x_{5,4} + x_{6,3}$;

  $x_{1,3} + x_{2,2} + x_{3,5} + x_{4,6} + x_{6,4} = x_{1,3} + x_{2,2} + x_{3,6} + x_{4,5} + x_{6,4}$; 

  $x_{1,3} + x_{2,1} + x_{3,5} + x_{4,6} + x_{6,4} = x_{1,3} + x_{2,1} + x_{3,6} + x_{4,5} + x_{6,4}$;

  $x_{1,1} + x_{2,2} + x_{3,5} + x_{4,6} + x_{6,4} = x_{1,1} + x_{2,2} + x_{3,6} + x_{4,5} + x_{6,4}$; 

  $x_{1,1} + x_{2,2} + x_{3,5} + x_{4,6} + x_{6,4} = x_{1,2} + x_{2,1} + x_{3,5} + x_{4,6} + x_{6,4}$;

  $x_{1,1} + x_{2,2} + x_{3,5} + x_{4,6} + x_{6,4} = x_{1,2} + x_{2,1} + x_{3,6} + x_{4,5} + x_{6,4}$; 

  $x_{1,1} + x_{2,2} + x_{3,5} + x_{4,6} + x_{6,3} = x_{1,1} + x_{2,2} + x_{3,6} + x_{4,5} + x_{6,3}$; 

  $x_{1,1} + x_{2,2} + x_{3,5} + x_{4,6} + x_{6,3} = x_{1,2} + x_{2,1} + x_{3,5} + x_{4,6} + x_{6,3}$; 

  $x_{1,1} + x_{2,2} + x_{3,5} + x_{4,6} + x_{6,3} = x_{1,2} + x_{2,1} + x_{3,6} + x_{4,5} + x_{6,3}$;

  $x_{1,1} + x_{2,2} + x_{3,3} + x_{4,6} + x_{6,4} = x_{1,2} + x_{2,1} + x_{3,3} + x_{4,6} + x_{6,4}$; 

  $x_{1,1} + x_{2,2} + x_{3,3} + x_{4,5} + x_{6,4} = x_{1,2} + x_{2,1} + x_{3,3} + x_{4,5} + x_{6,4}$;

  $x_{1,3} + x_{2,2} + x_{3,5} + x_{4,6} + x_{5,4} = x_{1,3} + x_{2,2} + x_{3,6} + x_{4,5} + x_{5,4}$; 

  $x_{1,3} + x_{2,1} + x_{3,5} + x_{4,6} + x_{5,4} = x_{1,3} + x_{2,1} + x_{3,6} + x_{4,5} + x_{5,4}$;

  $x_{1,1} + x_{2,2} + x_{3,5} + x_{4,6} + x_{5,4} = x_{1,2} + x_{2,1} + x_{3,5} + x_{4,6} + x_{5,4}$; 

  $x_{1,1} + x_{2,2} + x_{3,5} + x_{4,6} + x_{5,4} = x_{1,1} + x_{2,2} + x_{3,6} + x_{4,5} + x_{5,4}$;

  $x_{1,1} + x_{2,2} + x_{3,5} + x_{4,6} + x_{5,4} = x_{1,2} + x_{2,1} + x_{3,6} + x_{4,5} + x_{5,4}$; 

  $x_{1,1} + x_{2,2} + x_{3,5} + x_{4,6} + x_{5,3} = x_{1,2} + x_{2,1} + x_{3,5} + x_{4,6} + x_{5,3}$;

  $x_{1,1} + x_{2,2} + x_{3,5} + x_{4,6} + x_{5,3} = x_{1,1} + x_{2,2} + x_{3,6} + x_{4,5} + x_{5,3}$; 

  $x_{1,1} + x_{2,2} + x_{3,5} + x_{4,6} + x_{5,3} = x_{1,2} + x_{2,1} + x_{3,6} + x_{4,5} + x_{5,3}$;

  $x_{1,1} + x_{2,2} + x_{3,3} + x_{4,6} + x_{5,4} = x_{1,2} + x_{2,1} + x_{3,3} + x_{4,6} + x_{5,4}$; 

  $x_{1,1} + x_{2,2} + x_{3,3} + x_{4,5} + x_{5,4} = x_{1,2} + x_{2,1} + x_{3,3} + x_{4,5} + x_{5,4}$.
\end{center}
The linear space determined by these linear equations has dimension $33$. The linear equations coming from the $4 \times 4$ submatrices of the matrix
\begin{center}  
  $\left(\begin{array}{ccccccc} 1 & 1 & 0 & 1 & 0 & 0 & 0 \\ 1 & 0 & 1 & 0 & 0 & 0 & 1 \\ 0 & 1 & 0 & 0 & 0 & 1 & 1 \\ 1 & 0 & 0 & 0 & 1 & 1 & 0 \\ 0 & 0 & 0 & 1 & 1 & 0 & 1 \\ 0 & 0 & 1 & 1 & 0 & 1 & 0 \\ 0 & 1 & 1 & 0 & 1 & 0 & 0 \end{array}\right)$,
\end{center}
the symmetric version of the cocircuit matrix of the Fano matroid from \cite{dss}, are too numerous to be practical to list, but the linear space they determine has dimension $34$.

For symmetric matrices we can apply the same analysis. The only difference is the relation $x_{i,j} = x_{j,i}$ on the variables, and that the space of $n \times n$ symmetric matrices is, consequently, equivalent to $\mathbb{R}^{\binom{n}{2}}$. For the $6 \times 6$ symmetric matrix
\begin{center}  
  $\left(\begin{array}{ccccccc} 1 & 1 & 0 & 1 & 0 & 0 & 0 \\ 1 & 0 & 1 & 0 & 0 & 0 & 1 \\ 0 & 1 & 0 & 0 & 0 & 1 & 1 \\ 1 & 0 & 0 & 0 & 1 & 1 & 0 \\ 0 & 0 & 0 & 1 & 1 & 0 & 1 \\ 0 & 0 & 1 & 1 & 0 & 1 & 0 \\ 0 & 1 & 1 & 0 & 1 & 0 & 0 \end{array}\right)$,
\end{center}
we get the linear equations:
\begin{center}
  $x_{2,2} + 2x_{3,5} + 2x_{4,6} = x_{2,2} + x_{3,5} + x_{3,6} + x_{4,5} + x_{4,6}$;

  $x_{2,2} + 2x_{3,5} + 2x_{4,6} = x_{2,2} + 2x_{3,6} + 2x_{4,5}$; 

  $x_{1,2} + 2x_{3,5} + 2x_{4,6} = x_{1,2} + x_{3,5} + x_{3,6} + x_{4,5} + x_{4,6}$; 

  $x_{1,2} + 2x_{3,5} + 2x_{4,6} = x_{1,2} + 2x_{3,6} + 2x_{4,5}$;

  $x_{1,5} + x_{2,2} + x_{3,5} + 2x_{4,6} = x_{1,5} + x_{2,2} + 2x_{3,6} + x_{4,5}$; 

  $x_{1,5} + x_{2,2} + x_{3,5} + x_{3,6} + x_{4,6} = x_{1,5} + x_{2,2} + 2x_{3,6} + x_{4,5}$;

  $x_{1,3} + x_{2,2} + x_{3,5} + 2x_{4,6} = x_{1,3} + x_{2,2} + x_{3,6} + x_{4,5} + x_{4,6}$; 

  $x_{1,3} + x_{2,2} + x_{3,5} + x_{4,5} + x_{4,6} = x_{1,3} + x_{2,2} + x_{3,6} + 2x_{4,5}$;

  $x_{1,2} + 2x_{3,5} + 2x_{4,6} = x_{1,2} + x_{3,5} + x_{3,6} + x_{4,5} + x_{4,6}$; 

  $x_{1,2} + 2x_{3,5} + 2x_{4,6} = x_{1,2} + 2x_{3,6} + 2x_{4,5}$; 

  $x_{1,1} + 2x_{3,5} + 2x_{4,6} = x_{1,1} + x_{3,5} + x_{3,6} + x_{4,5} + x_{4,6}$; 

  $x_{1,1} + 2x_{3,5} + 2x_{4,6} = x_{1,1} + 2x_{3,6} + 2x_{4,5}$;

  $x_{1,2} + 2x_{3,5} + 2x_{4,6} = x_{1,2} + x_{1,5} + x_{3,6} + x_{4,5} + x_{4,6}$; 

  $x_{1,2} + x_{1,5} + x_{3,5} + x_{3,6} + x_{4,6} = x_{1,2} + x_{1,5} + 2x_{3,6} + x_{4,5}$;

  $x_{1,2} + x_{1,3} + x_{3,5} + 2x_{4,6} = x_{1,2} + x_{1,3} + x_{3,6} + x_{4,5} + x_{4,6}$; 

  $x_{1,2} + x_{1,3} + x_{3,5} + x_{4,5} + x_{4,6} = x_{1,2} + x_{1,3} + x_{3,6} + 2x_{4,5}$;

  $x_{1,3} + x_{2,2} + x_{3,5} + 2x_{4,6} = x_{1,3} + x_{2,2} + x_{3,6} + x_{4,5} + x_{4,6}$; 

  $x_{1,3} + x_{2,2} + x_{3,5} + 2x_{4,6} = x_{1,2} + x_{1,3} + x_{3,6} + x_{4,5} + x_{4,6}$;

  $x_{1,1} + x_{2,2} + 2x_{4,6} + x_{5,5} = 2x_{1,2} + 2x_{4,6} + x_{5,5}$; 

  $x_{1,1} + x_{2,2} + x_{3,6} + x_{4,6} + x_{5,5} = 2x_{1,2} + x_{3,6} + x_{4,6} + x_{5,5}$;

  $x_{1,1} + x_{2,2} + x_{3,5} + 2x_{4,6} = x_{1,1} + x_{2,2} + x_{3,6} + x_{4,5} + x_{4,6}$; 

  $x_{1,1} + x_{2,2} + x_{3,5} + 2x_{4,6} = 2x_{1,2} + x_{3,5} + 2x_{4,6}$;

  $x_{1,1} + x_{2,2} + x_{3,5} + 2x_{4,6} = 2x_{1,2} + x_{3,6} + x_{4,5} + x_{4,6}$; 

  $x_{1,1} + x_{2,2} + x_{3,5} + x_{4,5} + x_{4,6} = x_{1,1} + x_{2,2} + x_{3,6} + 2x_{4,5}$;

  $x_{1,1} + x_{2,2} + x_{3,5} + x_{4,5} + x_{4,6} = 2x_{1,2} + x_{3,5} + x_{4,5} + x_{4,6}$; 

  $x_{1,1} + x_{2,2} + x_{3,5} + x_{4,5} + x_{4,6} = 2x_{1,2} + x_{3,6} + 2x_{4,5}$;

  $x_{1,3} + x_{2,2} + x_{3,5} + x_{3,6} + x_{4,6} = x_{1,5} + x_{2,2} + 2x_{3,6} + x_{4,5}$; 

  $x_{1,2} + x_{1,5} + x_{3,5} + x_{3,6} + x_{4,6} = x_{1,2} + x_{1,5} + 2x_{3,6} + x_{4,5}$;

  $x_{1,1} + x_{2,2} + x_{3,6} + x_{4,6} + x_{5,5} = 2x_{1,2} + x_{3,6} + x_{4,6} + x_{5,5}$; 

  $x_{1,1} + x_{2,2} + 2x_{3,6} + x_{5,5} = 2x_{1,2} + 2x_{3,6} + x_{5,5}$;

  $x_{1,1} + x_{2,2} + x_{3,5} + x_{3,6} + x_{4,6} = 2x_{1,2} + x_{3,5} + x_{3,6} + x_{4,6}$; 

  $x_{1,1} + x_{2,2} + x_{3,5} + x_{3,6} + x_{4,6} = x_{1,1} + x_{2,2} + 2x_{3,6} + x_{4,5}$;

  $x_{1,1} + x_{2,2} + x_{3,5} + x_{3,6} + x_{4,6} = 2x_{1,2} + 2x_{3,6} + x_{4,5}$; 

  $x_{1,1} + x_{2,2} + 2x_{3,5} + x_{4,6} = 2x_{1,2} + 2x_{3,5} + x_{4,6}$;

  $x_{1,1} + x_{2,2} + 2x_{3,5} + x_{4,6} = x_{1,1} + x_{2,2} + x_{3,5} + x_{3,6} + x_{4,5}$; 

  $x_{1,1} + x_{2,2} + 2x_{3,5} + x_{4,6} = 2x_{1,2} + x_{3,5} + x_{3,6} + x_{4,5}$;

  $x_{1,3} + x_{2,2} + x_{3,5} + 2x_{4,6} = x_{1,3} + x_{2,2} + x_{3,6} + x_{4,5} + x_{4,6}$; 

  $x_{1,2} + x_{1,3} + x_{3,5} + 2x_{4,6} = x_{1,2} + x_{1,3} + x_{3,6} + x_{4,5} + x_{4,6}$;

  $x_{1,1} + x_{2,2} + x_{3,5} + 2x_{4,6} = x_{1,1} + x_{2,2} + x_{3,6} + x_{4,5} + x_{4,6}$; 

  $x_{1,1} + x_{2,2} + x_{3,5} + 2x_{4,6} = 2x_{1,2} + x_{3,5} + 2x_{4,6}$;

  $x_{1,1} + x_{2,2} + x_{3,5} + 2x_{4,6} = 2x_{1,2} + x_{3,6} + x_{4,5} + x_{4,6}$; 

  $x_{1,1} + x_{2,2} + x_{3,5} + x_{3,6} + x_{4,6} = x_{1,1} + x_{2,2} + 2x_{3,6} + x_{4,5}$; 

  $x_{1,1} + x_{2,2} + x_{3,5} + x_{3,6} + x_{4,6} = 2x_{1,2} + x_{3,5} + x_{3,6} + x_{4,6}$; 

  $x_{1,1} + x_{2,2} + x_{3,5} + x_{3,6} + x_{4,6} = 2x_{1,2} + 2x_{3,6} + x_{4,5}$;

  $x_{1,1} + x_{2,2} + x_{3,3} + 2x_{4,6} = 2x_{1,2} + x_{3,3} + 2x_{4,6}$; 

  $x_{1,1} + x_{2,2} + x_{3,3} + x_{4,5} + x_{4,6} = 2x_{1,2} + x_{3,3} + x_{4,5} + x_{4,6}$;

  $x_{1,3} + x_{2,2} + x_{3,5} + x_{4,5} + x_{4,6} = x_{1,3} + x_{2,2} + x_{3,6} + 2x_{4,5}$; 

  $x_{1,2} + x_{1,3} + x_{3,5} + x_{4,5} + x_{4,6} = x_{1,2} + x_{1,3} + x_{3,6} + 2x_{4,5}$;

  $x_{1,1} + x_{2,2} + x_{3,5} + x_{4,5} + x_{4,6} = 2x_{1,2} + x_{3,5} + x_{4,5} + x_{4,6}$; 

  $x_{1,1} + x_{2,2} + x_{3,5} + x_{4,5} + x_{4,6} = x_{1,1} + x_{2,2} + x_{3,6} + 2x_{4,5}$;

  $x_{1,1} + x_{2,2} + x_{3,5} + x_{4,5} + x_{4,6} = 2x_{1,2} + x_{3,6} + 2x_{4,5}$; 

  $x_{1,1} + x_{2,2} + 2x_{3,5} + x_{4,6} = 2x_{1,2} + 2x_{3,5} + x_{4,6}$;

  $x_{1,1} + x_{2,2} + 2x_{3,5} + x_{4,6} = x_{1,1} + x_{2,2} + x_{3,5} + x_{3,6} + x_{4,5}$; 

  $x_{1,1} + x_{2,2} + 2x_{3,5} + x_{4,6} = 2x_{1,2} + x_{3,5} + x_{3,6} + x_{4,5}$;

  $x_{1,1} + x_{2,2} + x_{3,3} + x_{4,5} + x_{4,6} = 2x_{1,2} + x_{3,3} + x_{4,5} + x_{4,6}$; 

  $x_{1,1} + x_{2,2} + x_{3,3} + 2x_{4,5} = 2x_{1,2} + x_{3,3} + 2x_{4,5}$.
\end{center}
The linear space determined by these linear equations has dimension $19$. Note also that the principal $4 \times 4$ submatrix with row/column indices $\{2,3,4,5\}$ is tropically singular, and not just symmetrically tropically singular. So, Lemma 8 can be applied.

\subsection{The dimension inequalities}

We now have everything we need in order to prove the main theorems of this paper. The first theorem is about standard matrices, their determinantal varieties, and the dimensions of the associated tropical varieties and tropical prevarieties.

\begin{thm}
  If $m_{1},\ldots,m_{s}$ are the $r \times r$ minors of an $m \times n$ matrix of variables, and $I_{m,n,r} = (m_{1},\ldots,m_{s})$ is the corresponding determinantal ideal, then the minors are a tropical basis if and only if the dimension of the tropical variety $\mathcal{T}(\textbf{V}(I_{m,n,r}))$ is equal to the dimension of the corresponding tropical prevariety $\cap_{i = 1}^{s}  \textbf{V}(\mathcal{T}(m_{i}))$.
\end{thm}

\begin{proof}
  Denote by $M_{m,n,r}$ the affine determinantal variety of $m \times n$ matrices with rank less than $r$. It is a standard result in algebraic geometry (\cite{h} Proposition 12.2, for example) that the dimension of $M_{m,n,r}$ is $(m+n-r+1)(r-1)$. It was proven in \cite{bjsst} that the tropical variety $\tilde{T}_{m,n,r}$ is a pure polyhedral fan with dimension equal to that of $M_{m,n,r}$.
  
  Using these formulas and our results from Section 5.3 we compute
  \begin{center}
    $dim(\tilde{T}_{6,6,5}) = (6+6-5+1)(5-1) = 32 < 33 \leq dim(T_{6,6,5})$,
  \end{center}
  and
  \begin{center}
    $dim(\tilde{T}_{7,7,4}) = (7+7-4+1)(4-1) = 33 < 34 \leq dim(T_{7,7,4})$.
  \end{center}

  Again, using these formulas we get
  \begin{center}
    $dim(\tilde{T}_{m+1,n,r}) - dim(\tilde{T}_{m,n,r}) = (m+n-r+2)(r-1) - (m+n-r+1)(r-1) = r-1$,
  \end{center}
  similarly,
  \begin{center}
    $dim(\tilde{T}_{m,n+1,r}) - dim(\tilde{T}_{m,n,r}) = (m+n-r+2)(r-1) - (m+n-r+1)(r-1) = r-1$,
  \end{center}
  and,
  \begin{center}
    $dim(\tilde{T}_{m+1,n+1,r+1}) - dim(\tilde{T}_{m,n,r}) - (m+n-r+2)r - (m+n-r+1)(r-1) = m+n+1$.
  \end{center}
  
  These, combined with Lemmas 5.7 and 5.8, prove that if $dim(\tilde{T}_{m,n,r}) < dim(T_{m,n,r})$ then
  \begin{center}
    $dim(\tilde{T}_{m+1,n,r}) < dim(T_{m+1,n,r})$,
   
    $dim(\tilde{T}_{m,n+1,r}) < dim(T_{m,n+1,r})$,

    $dim(\tilde{T}_{m+1,n+1,r+1}) < dim(T_{m+1,n+1,r+1})$.
  \end{center}

  From these results we may conclude $dim(\tilde{T}_{m,n,r}) < dim(T_{m,n,r})$ when $min(m,n) = 6$ and $r = 5$, or when $min(m,n) > 6$ and $4 \leq r < min(m,n)$. This covers all cases where the $r \times r$ minors do not form a tropical basis.
\end{proof}

We have a similar theorem for symmetric matrices.

\begin{thm}
  If $m_{1},\ldots,m_{s}$ are the $r \times r$ minors of an $n \times n$ symmetric matrix of variables, and $J_{n,r} = (m_{1},\ldots,m_{s})$ is the corresponding determinantal ideal, then for $4 < r < n$ the dimension of the tropical variety $\mathcal{T}(\textbf{V}(J_{n,r}))$ is less than the dimension of the corresponding tropical prevariety $\cap_{i = 1}^{s}  \textbf{V}(\mathcal{T}(m_{i}))$.
\end{thm}

\begin{proof}
  Denote by $Q_{n,r}$ the affine determinantal variety of symmetric $n \times n$ matrcies of rank less than $r$. It is a standard result in algebraic geometry (\cite{h} Chapter 22, Page 299) that the dimension of $Q_{n,r}$ is $(2nr-2n+3r-r^{2}-2)/2$. As in the previous theorem, the tropical variety $\tilde{S}_{n,r}$ is a pure polyhedral fan with dimension equal to that of $Q_{n,r}$.
  
  Using these formulas and our earlier result we compute
  \begin{center}
    $dim(\tilde{S}_{6,5}) = (60-12+15-25-2)/2 = 18 < 19 \leq dim(S_{6,5})$.
  \end{center}
  
  Again, using these formulas we get
  \begin{center}
    $dim(\tilde{S}_{n+1,r}) - dim(\tilde{S}_{n,r})$
    
    $= \frac{2(n+1)r-2(n+1)+3r-2-r^{2}}{2} - \frac{2nr-2n+3r-2-r^{2}}{2} = r-1$,
  \end{center}
  and,
  \begin{center}
    $dim(\tilde{S}_{n+1,r+1}) - dim(\tilde{S}_{n,r})$
    
    $= \frac{2(n+1)(r+1)-2(n+1)+3(r+1)-2-(r+1)^{2}}{2} - \frac{2nr-2n+3r-2-r^{2}}{2}$
    
    $= n+1$.
  \end{center}
  
  These, combined with Lemmas 5.9 and 5.10, prove that if $dim(\tilde{S}_{n,r}) < dim(S_{n,r})$, then
  \begin{center}
    $dim(\tilde{S}_{n+1,r}) < dim(S_{n+1,r})$,
    
    and
    
    $dim(\tilde{S}_{n+1,r+1}) < dim(S_{n+1,r+1})$.
  \end{center}
  
  From these results we may conclude $dim(\tilde{S}_{n,r}) < dim(S_{n,r})$ when $4 < r < n$.
\end{proof}


\begin{thebibliography}{b}

\bibitem[1]{bjsst}
  T. Bogart, A. Jensen, D. Speyer, B. Sturmfels, and R. Thomas,
  \newblock{\emph{Computing tropical varieties},}
  \newblock{J. Symbolic Comput. \textbf{42} (2007), nos. 1-2, 54-73.}

\bibitem[2]{c}
Melody Chan,
  \newblock{\emph{Research statement}.}
  \newblock{http://www.math.harvard.edu/\url{~}mtchan/.}

\bibitem[3]{cjr}
  Melody Chan, Anders Jensen, and Elena Rubei,
  \newblock{\emph{The $4 \times 4$ minors of a $5 \times n$ matrix are a tropical basis},}
  \newblock{Linear Algebra and its Applications. \textbf{435} (2011), 1590-1611.}
  
\bibitem[4]{dss}
  Mike Develin, Francisco Santos, and Bernd Sturmfels,
  \newblock{\emph{On the rank of a tropical matrix},}
  \newblock{Discrete and Computational Geometry, (E. Goodman, J.Pach and E.Welzl, eds.), MSRI Publicaitons, Cambridge Univ. Press, 2005.}

\bibitem[5]{h}
  Joe Harris,
  \newblock{\emph{Algebraic geometry: a first course},}
  \newblock{Graudate Texts in Mathematics, vol. 133, Springer, 1992.}

\bibitem[6]{ms}
  Diane Maclagan and Bernd Sturmfels,
  \newblock{\emph{Introduction to tropical geometry},}
  \newblock{Graduate Studies in Mathematics, vol. 161, American Mathematical Society, 2015.}

\bibitem[7]{sh1}
  Yaroslav Shitov,
  \newblock{\emph{Example of a 6-by-6 matrix with different tropical and Kapranov ranks}, (Russian, with English version available at arXiv:1012.5507v1.),}
  \newblock{Vestnik Moskov. Univ. Ser. 1. \textbf{5} (2011), 58-61.}

\bibitem[8]{sh2}
    Yaroslav Shitov,
    \newblock{\emph{When do the r-by-r minors of a matrix form a tropical basis?},}
    \newblock{Journal of Combinatorial Theory, Series A. \textbf{120} (2013), 1166-1201.}

\bibitem[9]{z1}
    Dylan Zwick,
    \newblock{\emph{Symmetric Kapranov and symmetric tropical ranks},}
    \newblock{preprint, arXiv:2112.14945[math.CO]}

\bibitem[10]{z2}
    Dylan Zwick,
    \newblock{\emph{The $4 \times 4$ minors of a $5 \times 5$ symmetric matrix are a tropical basis}}

\end{thebibliography}
\end{document}